\documentclass[reqno]{amsart}
\def\datum{July 20, 2020}
\def\chronologydatum{July 20, 2020}
\usepackage{amsthm, amscd, amsfonts,graphicx}
\usepackage{color}
\usepackage{enumerate}
\usepackage{verbatim}
\numberwithin{equation}{section}
\theoremstyle{plain}
 \newtheorem{theorem}{Theorem}[section]
 \newtheorem{lemma}[theorem]{Lemma}
 \newtheorem{corollary}[theorem]{Corollary}
\theoremstyle{definition}

 \newtheorem{remark}[theorem]{Remark}
\theoremstyle{remark}

\newcommand \Part[1] {\textup{Part}(#1)}
\newcommand \Equ[1] {\textup{Equ}(#1)}
\newcommand \sublat[1] {[#1]_{\kern-1pt\textup{lat}}}

\newcommand \lbrak {[\hskip-1.5pt[}
\newcommand \rbrak {\hskip0.5pt]\hskip-1.5pt]}
\newcommand \equ[2]{\lbrak{}#1,#2\rbrak{}^{\kern-0.5pt{\scriptscriptstyle\textup{e}}}}
\newcommand \kequ[1]{\lbrak{}#1\rbrak{}^{\kern-0.5pt{\scriptscriptstyle\textup{e}}}}
\newcommand \xequ[3]{\lbrak{}\langle #1,#2\rangle : #3\rbrak{}^{\kern-0.5pt{\scriptscriptstyle\textup{e}}}}
\newcommand \faequ[2]{\lbrak{}#1,#2\rbrak{}^{\kern-0.5pt{\scriptscriptstyle\textup{e}}}_{\kern-0.5pt\scriptscriptstyle\forall}}
\newcommand \bdelta{\delta^\bullet}
\newcommand \vonal {\noalign{\hrule}}
\newcommand \tuple [1] {\langle #1 \rangle}
\newcommand \pair [2] {\tuple{#1,#2}}
\newcommand \oal{\overline\alpha}
\newcommand \obe{\overline\beta}
\newcommand \oga{\overline\gamma}
\newcommand \ode{\overline\delta}
\newcommand \obmu{\overline{\pmb\mu}}

\newcommand \bmu{{\pmb\mu}}

\newcommand \eterm[2] {e_{#1,#2}}
\newcommand \fterm[2] {f_{#1,#2}}
\newcommand \gterm[3] {g^{(#1)}_{#2,#3}}
\newcommand \hterm[3] {h^{(#1)}_{#2,#3}}

\newcommand \enul {\Delta}
\newcommand \aHa {\mathcal H}
\newcommand \aGa {\mathcal G}
\newcommand \NN {{\mathbb N^+}}
\newcommand \maxs {\textup{Max}S}
\newcommand \Bell[1] {\textup{Bell}(#1)}
\newcommand \shd {{\delta^\sharp}}
\newcommand \vshd {{{\vec\delta\kern1pt}{}^\sharp}}
\newcommand \valpha{\vec\alpha}
\newcommand \vbeta{\vec\beta}
\newcommand \vgamma{\vec\gamma}

\newcommand \eeqref[1]{\overset{\eqref{#1}}{=}}
\newcommand \csm {{m_\ast}}
\newcommand \aph{Andr\'as}
\newcommand \ves {{\overline s}}
\newcommand \vew {{\overline w}}
\newcommand \ver {{\overline r}}
\newcommand \otau {{\overline \tau}}
\newcommand \length {\textup{length}}
\newcommand \distance[2] {\textup{distance}(#1,#2)}
\newcommand\set [1]{\{#1\}}
\newcommand \tbf[1] {\textbf{#1}}
\renewcommand \epsilon{\varepsilon}

\newcommand \red [1] {{\color{red}#1\color{black}}}

\newcommand \nothing [1] {}
\newcommand \magenta [1] {{\color{magenta}#1\color{black}}}

\begin{document}
\title[Partition lattices and authentication]
{Four-generated direct powers of partition lattices and authentication}

\author[G.\ Cz\'edli]{G\'abor Cz\'edli}
\address{University of Szeged, Bolyai Institute, Szeged,
Aradi v\'ertan\'uk tere 1, Hungary 6720}
\email{czedli@math.u-szeged.hu}
\urladdr{http://www.math.u-szeged.hu/~czedli/}

\date{\hfill {\tiny{\magenta{(\tbf{Always} check the author's website for possible updates!) }}}\  \red{\datum}}

\thanks{This research is supported by  NFSR of Hungary (OTKA), grant number K 134851}

\subjclass{06C10}

\keywords{Partition lattice, equivalence lattice, four-generated lattice,  Stirling number of the second kind, Bell number, secret key, authentication scheme, cryptography, crypto-system, commitment, semimodular lattice}

\dedicatory{Dedicated to Professor L\'aszl\'o Z\'adori on his sixtieth birthday}

\begin{abstract} For an integer $n\geq 5$, H.\ Strietz (1975) and  L.\ Z\'adori (1986) proved that
the lattice  $\Part n$  of all partitions of  $\set{1,2,\dots,n}$ is four-generated. Developing L.\ Z\'adori's particularly elegant construction further, we prove that even the $k$-th direct power $\Part n^k$ of $\Part n$ is four-generated for many but only finitely many exponents $k$. E.g., $\Part{100}^k$ is four-generated for every $k\leq 3\cdot 10^{89}$, and it has a four element generating set that is not an antichain for every $k\leq 1.4\cdot 10^{34}$.
In connection with these results, we outline a protocol how to use these lattices in authentication and secret key cryptography.
\end{abstract}

\maketitle

\section{Introduction}
This paper is dedicated to  L\'aszl\'o Z\'adori not only because of his birthday, but also because a nice  construction from his very first mathematical paper is heavily used here. 
Our starting point is that Strietz~\cite{strietz1,strietz2} proved in 1975 that 
\begin{equation}\left.
\parbox{7.5cm}{the lattice $\Part n$ of all partitions of the (finite)  set $\set{1,2,\dots,n}$ is a four-generated lattice.}
\,\,\right\}
\label{eqpbxstRszlT}
\end{equation} 
A decade later, Z\'adori~\cite{zadori} gave a very elegant proof of this result (and proved even more, which is not used in the present paper). Z\'adori's construction has opened lots of perspectives; this is witnessed by Chajda and Cz\'edli~\cite{chcz}, 
Cz\'edli~\cite{czedlismallgen,czedlifourgen,czedlioneonetwo,czgfourgeneqatoms},
Cz\'edli and Kulin~\cite{czgkulin},  Kulin~\cite{kulin}, and Tak\'ach~\cite{takach}. 

Our goal is to generalize \eqref{eqpbxstRszlT} from partition lattices to their direct powers; see Theorems~\ref{thmmain} and \ref{thmoot} later.
Passing from $\Part n$ to $\Part n^k$ has some content because of four reasons, which will be given with more details later;
here we only mention these reasons tangentially.
First, even the direct square of a four-generated lattice need not be four-generated. Second, if some direct power of a lattice is four-generated, then so are the original lattice and all of its other direct powers with smaller exponents; see Corollaries \ref{coroLprd} and \ref{corodjTd}.
Third, for each non-singleton finite lattice $L$, there is a (large) positive integer $k_0=k_0(L)$ such that for every $k\geq k_0$, the direct power $L^k$ is \emph{not} four-generated; this explains that the exponent is not arbitrary in our theorems. We admit that we could not 
determine the set $\set{k: \Part n^k\text{ is four-generated}}$, that is, we could not find 
the least $k_0$; this task will probably remain unsolved for long. Fourth, a whole section of this paper is devoted to the applicability of complicated lattices with few generators in Information Theory.

Although this paper has some links to Information Theory, it is primarily a \emph{lattice theoretical} paper. 
Note that only some elementary facts, regularly taught in graduate (and often in undergraduate) algebra, are needed about lattices. For those who know how to compute the join of two equivalence relations the paper is probably self-contained. If not, then  a small part of each of the monographs Burris and Sankappanavar \cite{burrsankapp},
Gr\"atzer~\cite{ggeneral,ggglt}, and Nation~\cite{nationbook} 
can be recommended; note that \cite{burrsankapp} and \cite{nationbook} are freely downloadable at the time of writing.

\subsection*{Outline}
The rest of the paper is structured as follows. Section~\ref{sectztrms}
gives the rudiments of partition lattices and recalls Z\'adori's construction in details; these details will be used in the subsequent two sections. Section~\ref{sectprod} formulates and prove our first result, Theorem~\ref{thmmain}, which
asserts that $\Part n^k$ is four-generated for certain values of $k$. In Section~\ref{sectootwo}, we formulate and prove Theorem~\ref{thmoot} about the existence of a four-element generating set of order type $1+1+2$ in $\Part n^k$.
Finally, Section~\ref{sectauth} offers a protocol  for authentication based on partition lattices and their direct powers; this protocol can also be used in secret key cryptography. 

%
%

\section{Rudiments and Z\'adori's construction}\label{sectztrms} 
Below, we are going to give some details in few lines for the sake of those not familiar with partition lattices and, in addition, we are going to fix the corresponding  notation.
For a set $A$, a set of pairwise disjoint nonempty subsets of $A$ is a \emph{partition} of $A$ if the union of these subsets, called \emph{blocks}, is $A$. For example, 
\begin{equation}
U=\set{\set{1,3},\set{2,4},\set{5}}
\label{eqUpPrrsB}
\end{equation}
is a partition of $A=\set{1,2,3,4,5}$. For pairwise distinct elements $a_1,\dots,a_k$ of $A$, the partition of $A$ with block $\set{a_1,\dots, a_k}$ such that all the other blocks are singletons will be denoted by
$\kequ{a_1,\dots a_k}$. Then, in our notation, $U$ from \eqref{eqUpPrrsB} is the same as
\begin{equation}
\equ13 + \equ24.
\label{eqczrggtbxmdGXP}
\end{equation}
For partitions $U$ and $V$ of $A$, we say that $U\leq V$ if and only if every block of $U$ is as subset of a (unique) block of $V$. With this ordering, the set of all partitions of $A$ turns into a lattice, which we denote by $\Part A$. 
For brevity, 
\begin{equation}
\text{$\Part n$ will stand for $\Part{\set{1,2,\dots,n}}$,}
\label{pbxPartnmM}
\end{equation}
and also for $\Part A$ when $A$ is a given set consisting of $n$ elements.  
 Associated with a partition $U$ of $A$, we define an \emph{equivalence relation} $\pi_U$ of $A$ as the collection of all pairs $(x,y)\in A^2$ such that $x$ and $y$ belong to the same block of $U$. As it is well known, the equivalence relations and the partitions of $A$ mutually determine each other, and $\pi_U\leq \pi_V$ (which is our notation for $\pi_U\subseteq \pi_V$) if and only if $U\leq V$. Hence, the \emph{lattice $\Equ A$ of all equivalence relations} of $A$ (in short, the \emph{equivalence lattice} of $A$) is isomorphic to $\Part A$. In what follows, we do not make a sharp distinction between a partition and the corresponding equivalence relation; no matter which of them is given, we can use the other one without warning. For example, \eqref{eqczrggtbxmdGXP} also denotes an equivalence relation associated with the partition given in \eqref{eqUpPrrsB}, provided the base set $\set{1,2,\dots,5}$ is understood. So we define and denote equivalences as the partitions above but we prefer to work in $\Equ A$ and $\Equ{n}=\Equ{\set{1,\dots,n}}$, because the lattice operations are easier to handle in $\Equ A$. For $\kappa,\lambda\in \Equ A$, the \emph{meet} and the \emph{join} of $\kappa$ and $\lambda$, denoted by $\kappa \lambda$ (or $\kappa\cdot\lambda$) and $\kappa+\lambda$, are the intersection and the transitive hull of the union of $\kappa$ and $\lambda$, respectively.  The advantage of this notation is that  the usual precedence rule allows us to write, say, $xy+xz$ instead of $(x\wedge y)\vee (x\wedge z)$.
\emph{Lattice terms} are composed from variables and join and meet operation signs in the usual way; for example, $f(x_1,x_2,x_3,x_4)=(x_1+x_2)(x_3+x_4)+(x_1+x_3)(x_2+x_4)$ is a quaternary lattice term. 
Given a lattice $L$ and $a_1,\dots, a_k\in L$, the \emph{sublattice generated} by $\set{a_1,\dots,a_k}$ is denoted and defined by 
\begin{equation}
\sublat{a_1,\dots,a_k}:=\set{f(a_1,\dots,a_k): a_1,\dots,a_k\in L,\,\,f\text{ is a lattice term}}.
\end{equation}
If there are pairwise distinct elements $a_1,\dots,a_k\in L$ such that $\sublat{a_1,\dots,a_k}=L$ then $L$ is said to be a \emph{$k$-generated lattice}.

Almost exclusively, we are going to define our equivalence relations  by (undirected simple, edge-coloured)  graphs. Every  horizontal thin straight edge is $\alpha$-colored but its color, $\alpha$, is not always indicated in the figures. The thin straight edges of slope 1, that is the southwest-northeast edges, are $\beta$-colored while the thin straight edges with slope $-1$, that is the southeast-northwest edges, are $\gamma$-colored. 
Finally, the thin \emph{solid} curved edges are $\delta$-colored. 
(We should disregard the \emph{dashed} ovals at this moment. Note that except for Figure~\ref{figd4}, every edge is thin.)
Figure~\ref{figd1} helps to keep this convention in mind. 
On the vertex set $A$, this figure and the other figures in the paper define an \emph{equivalence} (relation) $\alpha\in \Equ A$ in the following way: deleting all edges but the $\alpha$-colored ones, the components of the remaining graph are the blocks of the partition associated with $\alpha$. In other words, $\pair x y\in\alpha$ if and only if there is an $\alpha$-coloured path from vertex $x$ to vertex $y$ in the graph, that is, a path (of possibly zero length) all of whose edges are $\alpha$-colored. The equivalences $\beta$, $\gamma$, and $\delta$ are defined analogously. The success of Z\'adori's construction, to be discussed soon, lies in the fact of this visualization. Note that, to make our figures less crowded, the labels $\alpha,\dots,\delta$ are not always indicated but 
\begin{equation}
\text{our convention, shown in Figure \ref{figd1}, defines the colour of the edges} 
\label{eqtxtsdhClrsznzlD}
\end{equation}
even in this case.

\begin{figure}[htb] 
\centerline
{\includegraphics[scale=1.0]{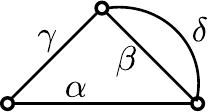}}
\caption{Standard notation for this paper
\label{figd1}}
\end{figure}%

Let us agree upon the following notation:
\begin{equation}\left.
\begin{aligned}
\sum_{\text{for all meaningful x}}&\equ{u_x}{v_x} \text{ will be denoted by }\cr
&\kern 2em\faequ{u_x}{v_x} \text{ or } \quad \faequ{u_y}{v_y};
\end{aligned}\,\right\}
\end{equation}
that is, each of $x$ and  $y$ in subscript or superscript position will mean that a join is formed for all meaningful values of these subscripts or superscript. If only a part of the meaningful subscripts or superscripts are needed in a join, then the following notational convention will be in effect:
\begin{equation}
\xequ{u^{(i)}}{v^{(i)}}{i\in I}\quad\text{ stands for }\quad
\sum_{i\in I}\equ{u^{(i)}}{v^{(i)}}.
\end{equation} 
For an integer $k\geq 2$ and the $(2k+1)$-element set 
\[Z=Z(2k+1):=\set{a_0,a_1,\dots,a_k, b_0,b_1,\dots, b_{k-1}},
\] 
we define
\begin{equation}
\begin{aligned}
&\alpha:=\kequ{a_0,a_1,\dots a_k} +\kequ{b_0,b_1,\dots b_{k-1}}=
\equ{a_x}{a_{x+1}}+\equ{b_y}{b_{y+1}}\cr
&\beta:=\faequ{a_x}{b_x}=\xequ{a_i}{b_i}{0\leq i\leq k-1},\cr
&\gamma:=\faequ{a_{x+1}}{b_x} = \xequ{a_{i+1}}{b_i}{0\leq i\leq k-1},\cr
&\delta:=\equ{a_0}{b_0}+\equ{a_k}{b_{k-1}}; 
\end{aligned}
\label{eqsBzTrGhxQ}
\end{equation}
see Figure~\ref{figd2}.
Then the system $\tuple{Z(2k+1);\alpha,\beta,\gamma,\delta}$ is called a $(2k+1)$-element \emph{Z\'adori configuration}. Its importance is revealed by the following lemma.

\begin{lemma}[Z\'adori~\cite{zadori}]\label{lemmazadori} For $k\geq 2$, 
$\sublat{\alpha,\beta,\gamma,\delta}=\Equ{Z(2k+1)}$, that is, the four partitions in \eqref{eqsBzTrGhxQ} of the Z\'adori configuration generate the lattice of all equivalences of $Z(2k+1)$. Consequently,
\begin{equation}
\sublat{\alpha,\beta,\gamma, \equ{a_0}{b_0}, \equ{a_k}{b_{k-1}}}=\Equ{Z(2k+1)}.
\label{eqzLhetzB}
\end{equation}
\end{lemma}

\begin{figure}[htb] 
\centerline
{\includegraphics[scale=1.0]{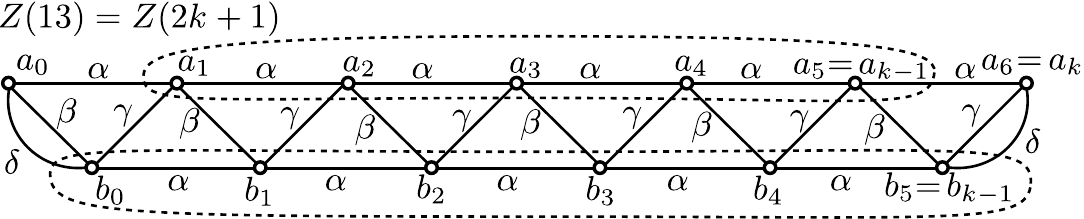}}
\caption{The Z\'adori configuration of odd size $2k+1$ with $k=6$
\label{figd2}}
\end{figure}%

We shall soon outline the proof of this lemma since we are going to use its details in the paper. But firstly, we formulate another lemma from Z\'adori \cite{zadori}, which has also been used in Cz\'edli \cite{czedlismallgen,czedlifourgen,czedlioneonetwo}
 and in other papers like Kulin~\cite{kulin}. We are going to recall its proof only for later reference.

\begin{lemma}[``Circle Principle'']\label{lemmaHamilt}
If $d_0,d_1,\dots,d_{n-1}$ are pairwise distinct elements of a set $A$ and $0\leq u<v\leq n-1$, then 
\begin{equation}\left.
\begin{aligned}
\equ {d_u}{d_v}=\bigl(\equ{d_{u}}{d_{u+1}} + \equ{d_{u+1}}{d_{u+2}}\dots + \equ{d_{v-1}}{d_{v}} \bigr) \cdot
 \bigl( \equ{d_{v}}{d_{v+1}} 
\cr
+ \dots + \equ{d_{n-2}}{d_{n-1}} +
\equ{d_{n-1}}{d_{0}}
 + \equ{d_{0}}{d_{1}}+ \dots + \equ{d_{u-1}}{d_{u}} \bigr)
\end{aligned}\,\right\}
\label{eqGbVrTslcdNssm}
\end{equation}
holds in $\Equ A$. 
If, in addition, $A=\set{d_0,d_1,\dots,d_{n-1}}$, then 
$\Equ A$ is generated by 
\[\set{\equ{d_{n-1}}{d_0} }\cup\bigcup_{0\leq i\leq n-2} \set{\equ{d_{i}}{d_i+1} }.
\]
\end{lemma}

\begin{proof}[Proof of Lemma~\ref{lemmaHamilt}] \eqref{eqGbVrTslcdNssm} is trivial. The second half of the lemma follows from the fact that for a finite $A$, the lattice $\Equ A$ is \emph{atomistic}, that is, each of its elements is the join of some atoms.
\end{proof}

\begin{proof}[Proof of Lemma~\ref{lemmazadori}]
On the set $\set{\oal,\obe,\oga,\ode}$ of variables, we are going to define several quaternary terms recursively. But first of all, we define the  quadruple
\begin{equation}
\obmu:=\tuple{\oal,\obe,\oga,\ode}
\label{eqMnPkBhJsZfPmF}
\end{equation}
of four variables with the purpose of abbreviating our quaternary terms $t(\oal,\obe,\oga,\ode)$  by $t(\obmu)$.
We let
\begin{equation}\left.
\begin{aligned}
g_0(\obmu)&:= \obe\,\ode\text{ (i.e.,}=\obe\wedge\ode), \cr
h_{i+1}(\obmu)&:=((g_i(\obmu)+\oga)\oal+g_i(\obmu))\oga\text{ for }i\geq 0,\cr
g_{i+1}(\obmu)&:=((h_{i+1}(\obmu)+\obe)\oal+h_{i+1}(\obmu))\obe\text{ for }i\geq 0,\cr
H_0(\obmu)&:=\oga\ode,\cr
G_{i+1}(\obmu)&:=((H_i(\obmu)+\obe)\oal + H_i(\obmu))\obe \text{ for }i\geq 0,\cr
H_{i+1}(\obmu)&:=((G_{i+1}(\obmu)+\oga)\oal+G_{i+1}(\obmu))\oga  \text{ for }i\geq 0.
\end{aligned}
\,\,\right\}
\label{eqZhgRsMks}
\end{equation}
For later reference, let us point out that
\begin{equation}\left.
\parbox{5.4cm}{in \eqref{eqZhgRsMks}, $\delta$ is used only twice: to define $g_0(\obmu)$ and to define $H_0(\obmu)$.}
\,\,\right\}
\label{eqpbxTwCzWn}
\end{equation}
Next, in harmony with \eqref{eqsBzTrGhxQ} and Figure~\ref{figd2},  we let 
\begin{equation}
\bmu:=\tuple{\alpha,\beta,\gamma,\delta}.
\label{eqMdhzBfTnQslwvP}
\end{equation}
Clearly,
\begin{equation}
\beta\delta=\equ{a_0}{b_0}\,\,\text{ and }\,\,\gamma\delta=\equ{a_k}{b_{k-1}}.
\label{eqNTjgMzdD}
\end{equation}
An easy induction shows that
\begin{equation}\left.
\begin{aligned}
g_i(\bmu)&:=\xequ{a_j}{b_j}{0\leq j\leq i} \text{ for }0\leq i\leq k-1, \cr
h_{i}(\bmu)&:=\xequ{a_j}{b_{j-1}}{1\leq j\leq i}\text{ for }1\leq i\leq k,\cr
H_i(\bmu)&:=\xequ{a_{k-j}}{b_{k-1-j}}{0\leq j\leq i} \text{ for }0\leq i\leq k-1, \cr
G_{i}(\bmu)&:=\xequ{a_{k-j}}{b_{k-j}}{1\leq j\leq i}\text{ for }1\leq i\leq k.\cr
\end{aligned}
\,\,\right\}
\label{eqmlZrbQshPrk}
\end{equation}
Next, for certain edges $\pair u v$ of the graph given in Figure~\ref{figd2}, we define a corresponding lattice term $\eterm u v(\obmu)$ as follows.
\begin{equation}\left.
\begin{aligned}
\eterm{a_i}{b_i}(\obmu)&:=g_i(\obmu)\cdot G_{k-i}(\obmu),\quad
\text{for }0\leq i\leq k-1,\cr
\eterm{a_i}{b_{i-1}}(\obmu)&:=h_i(\obmu)\cdot H_{k-i}(\obmu),\quad \text{for }1\leq i\leq k\cr
\eterm{a_i}{a_{i+1}}(\obmu)&:=\oal\cdot (\eterm{a_i}{b_i}(\obmu) + \eterm{a_{i+1}}{b_{i}}(\obmu)),\quad \text{for }0\leq i\leq k-1,\cr 
\eterm{b_i}{b_{i+1}}(\obmu)&:=\oal\cdot (\eterm{a_{i+1}}{b_i}(\obmu) + \eterm{a_{i+1}}{b_{i+1}}(\obmu)),\quad 0\leq i\leq k-2.
\end{aligned}
\,\right\}
\label{eqmlZsknNgRvTk}
\end{equation}
The first two equalities below follow from \eqref{eqmlZrbQshPrk}, while 
the third and the fourth from the first two.
\begin{equation}\left.
\begin{aligned}
\eterm{a_i}{b_i}(\bmu)&=\equ{a_i}{b_i},\quad
\text{for }0\leq i\leq k-1,\cr
\eterm{a_i}{b_{i-1}}(\bmu)&=\equ{a_i}{b_{i-1}}\quad \text{for }1\leq i\leq k\cr
\eterm{a_i}{a_{i+1}}(\bmu)&=
\equ{a_i}{a_{i+1}}\quad \text{for }0\leq i\leq k-1,\cr 
\eterm{b_i}{b_{i+1}}(\bmu)&=
\equ{b_i}{b_{i+1}}, \quad 0\leq i\leq k-2.
\end{aligned}
\,\right\}
\label{eqmlspnSzdRkMblk}
\end{equation}
Finally, let 
\begin{equation}\tuple{d_0,d_1,\dots,d_{n-1}}:=\tuple{a_0,a_1,\dots,a_k,b_{k-1},b_{k-2},\dots, b_0}.
\label{eqdsrzLmJnszpkJvStT}
\end{equation}  
In harmony with \eqref{eqGbVrTslcdNssm}, we define the following term
\begin{equation}\left.
\begin{aligned}
\eterm{d_u}{d_v}(\obmu):=\bigl(\eterm{d_{u}}{d_{u+1}}(\obmu) + \eterm{d_{u+1}}{d_{u+2}}(\obmu)\dots + \eterm{d_{v-1}}{d_{v}}(\obmu) \bigr) \cdot
 \bigl( \eterm{d_{v}}{d_{v+1}}(\obmu) 
\cr
+ \dots + \eterm{d_{n-2}}{d_{n-1}}(\obmu) +
\eterm{d_{n-1}}{d_{0}}
 + \eterm{d_{0}}{d_{1}}(\obmu)+ \dots + \eterm{d_{u-1}}{d_{u}}(\obmu) \bigr)
\end{aligned}\,\right\}
\label{eqnbVrtzstVcX}
\end{equation}
for $0\leq u< v\leq n-1= 2k$. Combining \eqref{eqGbVrTslcdNssm},  \eqref{eqmlspnSzdRkMblk}, and \eqref{eqnbVrtzstVcX}, we obtain that
\begin{equation}
\eterm{d_u}{d_v}(\bmu)= \equ{d_u}{d_v}.
\label{eqczhhndPMtlvRdwdb}
\end{equation}
Based on \eqref{eqpbxTwCzWn}, note at this point that in \eqref{eqZhgRsMks}, \eqref{eqmlZsknNgRvTk}, and \eqref{eqnbVrtzstVcX}, $\delta$ is used only twice: to define $g_0(\obmu)$ and to define $H_0(\obmu)$. Consequently, taking \eqref{eqNTjgMzdD} also into account, we conclude that 
\begin{equation}\left.
\parbox{9.0cm}{equality \eqref{eqczhhndPMtlvRdwdb} remains valid if $\delta$, the fourth component of $\bmu$,  is replaced by any other partition whose meet with $\beta$ and that with $\gamma$ are $\equ {a_0}{b_0}$ and  $\equ {a_k}{b_{k-1}}$, respectively.} \,\,\right\}
\label{eqpbxZbntGhSwD}
\end{equation}
Since every atom of $\Equ{Z(2k+1)}$ is of the form \eqref{eqczhhndPMtlvRdwdb} and $\Equ{Z(2k+1)}$ is an atomistic lattice, $\sublat{\alpha,\beta,\gamma,\delta}=\Equ{Z(n)}$.
In virtue of \eqref{eqpbxZbntGhSwD} and since $\ode$ has been used only twice, \eqref{eqzLhetzB} also holds, completing the proof of Lemma~\ref{lemmazadori}.
\end{proof}

\begin{figure}[htb] 
\centerline
{\includegraphics[scale=1.0]{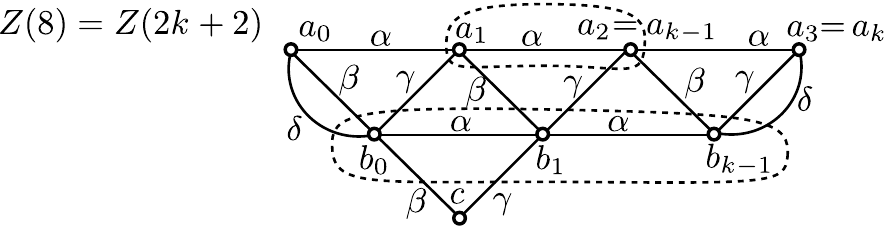}}
\caption{A  configuration for even size $2k+2$ with $k=3$
\label{figd3}}
\end{figure}%

Next, for $k\geq 2$,  we add a new vertex $c$, a $\beta$-colored edge $\pair{b_0}{c}$, and a $\gamma$-colored edge  $\pair{b_{2}}{c}$
to $Z(2k+1)$ to obtain $Z(2k+2)$, see Figure~\ref{figd3}. This configuration is different from what Z\'adori~\cite{zadori} used for the even case; our approach by Figure~\ref{figd3} is simpler and fits better to our purposes. Again, the dashed curved edges of Figure~\ref{figd3} should be disregarded until otherwise is stated.

\begin{lemma}\label{lemmazeven}
For $n=2k+2\geq 6$, we have that $\Equ{Z(n)}=\sublat{\alpha,\beta,\gamma, \delta}$.
\end{lemma}

\begin{proof} With the short terms $\oal^\ast{}=\oal$, $\obe^\ast{}:=\obe(\oal+\ode)$, $\oga^\ast{}:=\oga(\oal+\ode)$, and $\ode^\ast{}:=\ode$, we define
$\obmu^\ast{}:=\tuple{\oal^\ast{},\obe^\ast{},\oga^\ast{},\ode^\ast{}}$. For each term $t$ defined in \eqref{eqZhgRsMks} and \eqref{eqmlZsknNgRvTk}, we define a term $t^\ast{}$ as 
$t^\ast{}(\obmu):=t(\obmu^\ast{})$.
We also need the corresponding partitions $\alpha^\ast{}:=\alpha$, $\beta^\ast{}:=\beta(\alpha+\delta)$, $\gamma^\ast{}:=\gamma(\alpha+\delta)$, 
$\delta^\ast{}:=\delta$, and the quadruple $\bmu^\ast{}:=\tuple{\alpha^\ast{},\beta^\ast{},\gamma^\ast{},\delta^\ast{}}$.
Apart from the singleton block $\set c$,
they are the same as the partitions considered in Lemma~\ref{lemmazadori} for $Z(2k+1)$. Hence, it follows that 
\eqref{eqmlZrbQshPrk}, \eqref{eqmlspnSzdRkMblk}, and \eqref{eqczhhndPMtlvRdwdb} hold with $\bmu^\ast{}$ instead of $\bmu$.
In other words, they hold with $\bmu$ if the terms $t$ are replaced by the corresponding terms $t^\ast$. 
In particular,  \eqref{eqczhhndPMtlvRdwdb} is reworded as follows:
\begin{equation}
\eterm x y^\ast(\bmu)=\equ x y\quad\text{ for all }x,y\in Z(n)\setminus\set c.
\label{eqmsSzvKlJjJrlcsPsm}
\end{equation}
So if we define (without defining their ``asterisk-free versions'' $\eterm{a_0}{c}$ and $\eterm{a_2}{c}$) the terms
\begin{equation}
\eterm{a_0}{c}^\ast(\obmu):=\obe\cdot\bigl(\oga+\eterm{a_0}{a_{2}}(\obmu^\ast{}) \bigr)\text{ and } 
\eterm{a_{2}}{c}^\ast(\obmu):=\oga\cdot\bigl(\obe+\eterm{a_0}{a_{2}}(\obmu^\ast{}) \bigr),
\label{eqdzltGmrkszTnldpfGl}
\end{equation}
then it follows easily that 
\begin{equation}
\eterm{a_0}{c}^\ast(\bmu)= \equ{a_0}{c}\,\,\text{ and }\,\,
 \eterm{a_2}{c}^\ast(\bmu)= \equ{a_2}{c}; 
\label{eqhfnWhtjCxT}
\end{equation}
remark that in addition to \eqref{eqmsSzvKlJjJrlcsPsm},  \eqref{eqhfnWhtjCxT} also belongs to the scope of \eqref{eqpbxZbntGhSwD}.
Let 
\begin{equation}
\tuple{d_0,d_1,\dots,d_{n-1}}
:=\tuple{a_0,c,a_2,a_3,\dots,a_{k},b_{k-1},b_{k-2},\dots,b_1,a_1,b_0}.
\label{eqchtFkJknFsgsN}
\end{equation}
Similarly to \eqref{eqnbVrtzstVcX} but now
based on \eqref{eqchtFkJknFsgsN} rather than  \eqref{eqdsrzLmJnszpkJvStT}, we define the following term (without defining its ``non-asterisked'' $\fterm{d_u}{d_v}$ version)
\begin{equation}\left.
\begin{aligned}
\fterm{d_u}{d_v}^\ast (\obmu):=\bigl(\eterm{d_{u}}{d_{u+1}}^\ast (\obmu) + \eterm{d_{u+1}}{d_{u+2}}^\ast (\obmu)\dots + \eterm{d_{v-1}}{d_{v}}^\ast (\obmu) \bigr) \cdot
 \bigl( \eterm{d_{v}}{d_{v+1}}^\ast (\obmu) 
\cr
+ \dots + \eterm{d_{n-2}}{d_{n-1}}^\ast (\obmu) +
\eterm{d_{n-1}}{d_{0}}
 + \eterm{d_{0}}{d_{1}}^\ast (\obmu)+ \dots + \eterm{d_{u-1}}{d_{u}}^\ast (\obmu) \bigr)
\end{aligned}\,\right\}
\label{eqnvlqLY}
\end{equation}
for $0\leq u<v < n$.  By Lemma~\ref{lemmaHamilt}, \eqref{eqmsSzvKlJjJrlcsPsm}, \eqref{eqhfnWhtjCxT}, and \eqref{eqnvlqLY},  we obtain that 
\begin{equation}
\fterm x y^\ast(\bmu)=\equ x y\quad\text{ for all }x\neq y\in Z(n).
\label{eqnmsdsulTVcX}
\end{equation}
The remark right after \eqref{eqhfnWhtjCxT} allows us to note that 
\begin{equation}
\text
{\eqref{eqnmsdsulTVcX} also belongs to the scope of \eqref{eqpbxZbntGhSwD}.}
\label{eqnZhgjdlSrkhllRw}
\end{equation}
Finally, \eqref{eqnmsdsulTVcX} implies Lemma~\ref{lemmazeven} since $\Equ{Z(n)}$ is atomistic. 
\end{proof}

\section{Generating direct powers of partition lattices}\label{sectprod}
Before formulating the main result of the paper,  we recall some notations and concepts. The lower integer part of a real number $x$ will be denoted by $\lfloor x\rfloor$; for example, 
$\lfloor \sqrt 2\rfloor=1$ and $\lfloor  2\rfloor=2$. The set of positive integer numbers will be denoted by $\NN$. For $n\in\NN$, the number of partitions of the $n$-element set $\set{1,2,\dots,n}$, that is, the size of $\Part n\cong \Equ n$ is the so-called $n$-th \emph{Bell number}; it will be denoted by $\Bell n$.
The number of partitions of $n$ objects with exactly $r$ blocks is denoted by $S(n,r)$; it is the \emph{Stirling number of the second kind} with parameters $n$ and $r$. Note that $S(n,r)\geq 1$ if and only if $1\leq r\leq n$; otherwise $S(n,r)$ is zero. Clearly, $\Bell n=S(n,1)+S(n,2)+\dots +S(n,n)$. 
Let 
\begin{equation}
\text{$\maxs(n)$ denote the maximal element  of the set $\set{S(n,r): r\in\NN}$.}
\label{eqtxtmMxSdfLm}
\end{equation}
We know from 
Rennie and Dobson~\cite[page 121]{renniedobson} that
\begin{equation}
\log \maxs(n)= n \log n - n\log \log n - n + O\left( n\cdot {\frac{\log\log n}{\log n}} \right).
\end{equation}
Hence, $\maxs(n)$ is quite large; see Tables~\eqref{tablerdDbsa}--\eqref{tablerdDbsc} and \eqref{tablerdDbsg} for some of its values; note that those given in exponential form are only rounded values. 
Some rows occurring in these tables, computed by Maple V. Release 5 (1997) under Windows 10, will be explained later.

\allowdisplaybreaks{
\begin{align}
&
\lower  0.8 cm
\vbox{\tabskip=0pt\offinterlineskip
\halign{\strut#&\vrule#\tabskip=1pt plus 2pt&
#\hfill& \vrule\vrule\vrule#&
\hfill#&\vrule#&
\hfill#&\vrule#&
\hfill#&\vrule#&
\hfill#&\vrule#&
\hfill#&\vrule#&
\hfill#&\vrule#&
\hfill#&\vrule#&
\hfill#&\vrule#&
\hfill#&\vrule#&
\hfill#&\vrule#&
\hfill#&\vrule#&
\hfill#&\vrule\tabskip=0.1pt#&
#\hfill\vrule\vrule\cr
\vonal\vonal\vonal\vonal
&&\hfill$n$&&$\,1$&&$\,2$&&$\,3$&&$\,4$&&$5$&&$6$&&$7$&&$8$&&$9$&&$10$&&$11$&&$12$&
\cr\vonal\vonal
&&$\maxs(n)$&&$1$&&$1$&&$3$&&$7$&&$25$&&$90$&&$350$&&$1\,701$&&$7\,770$&&$42\,525$&&$246\,730$&&$1\,379\,400$&
\cr\vonal
&&\hfill$m(n)$&&$\phantom a$&&$\phantom b$&&$\phantom c$&&$\phantom d$&&$1$&&$1$&&$3$&&$3$&&$21$&&$21$&&$175$&&$175$&\cr
\vonal
&&\hfill$\csm(n)$&&$\phantom a$&&$\phantom b$&&$\phantom c$&&$\phantom d$&&$\phantom d$&&$\phantom e$&&$1$&&$1$&&$1$&&$1$&&$2$&&$2$&\cr
\vonal\vonal\vonal\vonal
}} 
\label{tablerdDbsa}
\\
&
\lower  0.8 cm
\vbox{\tabskip=0pt\offinterlineskip
\halign{\strut#&\vrule#\tabskip=1pt plus 2pt&
#\hfill& \vrule\vrule\vrule#&
\hfill#&\vrule#&
\hfill#&\vrule#&
\hfill#&\vrule#&
\hfill#&\vrule#&
\hfill#&\vrule\tabskip=0.1pt#&
#\hfill\vrule\vrule\cr
\vonal\vonal\vonal\vonal
&&\hfill$n$&&$13$&&$14$&&$15$&&$16$&&$17$&
\cr\vonal\vonal
&&$\maxs(n)$&&$9\,321\,312$&&$63\,436\,373$&&$420\,693\,273$&&$3\,281\,882\,604$&&$25\,708\,104\,786$&
\cr
\vonal
&&\hfill $m(n)$&&$2\,250$&&$2\,250$&&$31\,500$&&$31\,500$&&$595\,350$&
\cr
\vonal
&&\hfill $\csm(n)$&&$2$&&$2$&&$9$&&$9$&&$9$&
\cr
\vonal\vonal\vonal\vonal
}}
\label{tablerdDbsb} 
\\
&
\lower  0.8 cm
\vbox{\tabskip=0pt\offinterlineskip
\halign{\strut#&\vrule#\tabskip=1pt plus 2pt&
#\hfill& \vrule\vrule\vrule#&
\hfill#&\vrule#&
\hfill#&\vrule#&
\hfill#&\vrule\tabskip=0.1pt#&
#\hfill\vrule\vrule\cr
\vonal\vonal\vonal\vonal
&&\hfill$n$&&$18$&&$19$&&$20$&
\cr\vonal\vonal
&&$\maxs(n)$&&$1\,974\,624\,834\,000$&&$1\,709\,751\,003\,480$&&$15\,170\,932\,662\,679$&
\cr\vonal
&&\hfill$m(n)$&&$595\,350$&&$13\,216\,770$&&$13\,216\,770$&
\cr\vonal
&&\hfill$\csm(n)$&&$9$&&$49$&&$49$&
\cr\vonal\vonal\vonal\vonal
}}
\label{tablerdDbsc} 
\\
&
\lower  0.8 cm
\vbox{\tabskip=0pt\offinterlineskip
\halign{\strut#&\vrule#\tabskip=1pt plus 2pt&
#\hfill& \vrule\vrule\vrule#&
\hfill#&\vrule#&
\hfill#&\vrule#&
\hfill#&\vrule#&
\hfill#&\vrule#&
\hfill#&\vrule\tabskip=0.1pt#&
#\hfill\vrule\vrule\cr
\vonal\vonal\vonal\vonal
&&\hfill$n$&&$21$&&$22$&&$23$&&$24$&&$25$&
\cr\vonal\vonal
&&\hfill$m(n)$&&$330\,419\,250$&&$330\,419\,250$&&$10\,492\,193\,250 $&&$10\,492\,193\,250 $&& $ 3.40\cdot 10^{11} $ &
\cr\vonal
&&\hfill$\csm(n)$&&$49$&&$49$&&$625$&&$625 $&& $625 $ &
\cr\vonal\vonal\vonal\vonal
}}
\label{tablerdDbsd} 
\\
%
%
&
\lower  0.8 cm
\vbox{\tabskip=0pt\offinterlineskip
\halign{\strut#&\vrule#\tabskip=1pt plus 2pt&
#\hfill& \vrule\vrule\vrule#&
\hfill#&\vrule#&
\hfill#&\vrule#&
\hfill#&\vrule#&
\hfill#&\vrule#&
\hfill#&\vrule#&
\hfill#&\vrule\tabskip=0.1pt#&
#\hfill\vrule\vrule\cr
\vonal\vonal\vonal\vonal
&&\hfill$n$&&$26$&&$27$&&$28$&&$29$&&$30$&&$31$&
\cr\vonal\vonal
&&\hfill$m(n)$&&$3.40\cdot 10^{11}$&&$1.29\cdot 10^{13}$&&$1.29\cdot 10^{13}$&&$5.91\cdot 10^{14}$&& $5.91\cdot 10^{14}$ && $2.67\cdot 10^{16}$ &
\cr\vonal
&&\hfill$\csm(n)$&&$625$&&$8100$&&$8100$&&$8100$&& $8100$ && $122500 $ &
\cr\vonal\vonal\vonal\vonal
}}
\label{tablerdDbse} 
\\
%
%
&
\lower  0.8 cm
\vbox{\tabskip=0pt\offinterlineskip
\halign{\strut#&\vrule#\tabskip=1pt plus 2pt&
#\hfill& \vrule\vrule\vrule#&
\hfill#&\vrule#&
\hfill#&\vrule#&
\hfill#&\vrule#&
\hfill#&\vrule#&
\hfill#&\vrule#&
\hfill#&\vrule\tabskip=0.1pt#&
#\hfill\vrule\vrule\cr
\vonal\vonal\vonal\vonal
&&\hfill$n$&&$32$&&$33$&&$34$&&$35$&&$36$&&$37$&
\cr\vonal\vonal
&&\hfill$m(n)$&&$2.67\cdot 10^{16}$&&$1.38\cdot 10^{18}$&&$1.38\cdot 10^{18}$&&$8.44\cdot 10^{19}$&& $8.44\cdot 10^{19}$ && $5.08\cdot 10^{21}$ &
\cr\vonal
&&\hfill$\csm(n)$&&$122500$&&$122500$&&$122500$&&$2893401$&& $2893401$ && $2893401$ &
\cr\vonal\vonal\vonal\vonal
}}
\label{tablerdDbsf} 
\\
%
&
\lower  0.8 cm
\vbox{\tabskip=0pt\offinterlineskip
\halign{\strut#&\vrule#\tabskip=1pt plus 2pt&
#\hfill& \vrule\vrule\vrule#&
\hfill#&\vrule#&
\hfill#&\vrule#&
\hfill#&\vrule#&
\hfill#&\vrule#&
\hfill#&\vrule\tabskip=0.1pt#&
#\hfill\vrule\vrule
\cr\vonal\vonal\vonal\vonal
&&\hfill$n$&&$97$&&$98$&&$99$&&$100$&&$2020$&
\cr\vonal\vonal
&&\hfill$\maxs(n)$&&$3.22\cdot 10^{110}$&&$9.31\cdot 10^{111}$&&$2.69\cdot 10^{113}$&& $7.77\cdot 10^{114}$ &&$3.81\cdot10^{4398}$&
\cr\vonal
&&\hfill$m(n)$&&$1.08\cdot 10^{87}$&&$1.08\cdot 10^{87}$&&$3.09\cdot 10^{89}$&&$3.09\cdot 10^{89}$&&$5.52\cdot 10^{3893}$&
\cr\vonal
&&\hfill$\csm(n)$&&$1.52\cdot 10^{32}$&&$1.52\cdot 10^{32}$&&$1.45\cdot 10^{34}$&& $1.45\cdot 10^{34}$ &&$3.97\cdot 10^{1700}$&
\cr\vonal\vonal\vonal\vonal
}}
\label{tablerdDbsg} 
\end{align}

The aim of this section is to prove the following theorem; \eqref{pbxPartnmM} and \eqref{eqtxtmMxSdfLm} are still in effect.

\begin{theorem}\label{thmmain} Let $n\geq 5$ be an integer, let 
$k:=\lfloor (n-1)/2 \rfloor$, and let 
\begin{equation}
m=m(n):=\maxs(k)\cdot \maxs(k-1).
\label{eqmSrpRdgB}
\end{equation}
Then $\Part n^m$ or, equivalently, $\Equ n^m$ is four-generated. In other words, the $m$-th direct power of the lattice of all partitions of the set $\set{1,2,\dots, n}$ is generated by a four-element subset.
\end{theorem}

Some values of $m(n)$ are given in Tables~\eqref{tablerdDbsa}--\eqref{tablerdDbsg}. Before proving this theorem, we formulate some remarks and  corollaries and we make some comments.

\begin{corollary}\label{coroLprd}
Let $n$ and $m$ as in Theorem \ref{thmmain}. Then for every integer $t$ with $1\leq t\leq m$, the direct power $\Part n^t$ is four-generated. In particular, $\Part n$ in itself is four-generated.
\end{corollary}

The second half of Corollary~\ref{coroLprd} shows that Theorem~\ref{thmmain} is a stronger statement than the Strietz--Z\'adori result; see \eqref{eqpbxstRszlT} in the Introduction.  This corollary follows quite easily from Theorem~\ref{thmmain} as follows.

\begin{proof}[Proof of Corollary~\ref{coroLprd}]
Since the natural projection $\Part n^m \to \Part n^t$, defined by $\tuple{x_1,\dots, x_m}\mapsto \tuple{x_1,\dots, x_t}$, sends a 4-element generating set into an at most 4-element generating set, Theorem~\ref{thmmain} applies.
\end{proof}
 
\begin{remark} We cannot say that $m=m(n)$ in Theorem~\ref{thmmain} is the largest possible exponent. First, because the proof that we are going to present relies on a particular construction and we do not know whether there exist  better constructions for this purpose. 
Second, because we use Stirling numbers of the second kind to give a lower estimate of the size of a maximum-sized antichain in partition lattices, and we know from 
Canfield~\cite{canfield} that this estimate is not sharp. However, this fact would not lead to a reasonably esthetic improvement of Theorem~\ref{thmmain}.
\end{remark}

\begin{remark}\label{remnVrbG} If $n$ and $t$ are positive integers such that $n\geq 4$ and
\begin{equation}
t >  \Bell n \cdot \Bell{n-1}\cdot \Bell{n-2}\cdot \Bell{n-3} ,
\label{eqthmBsrhzTq}
\end{equation} 
then $\Part n^t$ is not four-generated. Thus, the exponent in Theorem~\ref{thmmain} cannot be arbitrarily large.
\end{remark}

The product occurring in \eqref{eqthmBsrhzTq} is much larger than $m(n)$ in \eqref{eqmSrpRdgB}. Hence, there is a wide interval of integers $t$ such that we do not know whether $\Part n^t$ is four-generated or not.

\begin{proof}[Proof of Remark~\ref{remnVrbG}]
Let $p$ denote the product in  \eqref{eqthmBsrhzTq}.
For the sake of contradiction, suppose that $t>p$ but $\Part n^t$ is generated by some $\set{\alpha,\beta,\gamma,\delta}$. Here $\alpha=\tuple{\alpha_1,\alpha_2,\dots,\alpha_t}$
with all the $\alpha_i\in\Part n$, and similarly for $\beta$, $\gamma$, and $\delta$. By the easy argument proving Corollary~\ref{coroLprd}, we know that $\set{\alpha_i,\beta_i,\gamma_i,\delta_i}$ generates $\Part n$ for all $i\in\set{1,\dots, t}$. 
Since $\Part n$ is not 3-generated by Z\'adori~\cite{zadori}, the quadruple $\tuple{\alpha_i,\beta_i,\gamma_i,\delta_i}$ consists of pairwise distinct components. But there are only $p$ such quadruples, whereby the  the pigeonhole principle yields two distinct subscripts $i$ and $j$  $\in\set{1,\dots, t}$
such that $\tuple{\alpha_i,\beta_i,\gamma_i,\delta_i}=\tuple{\alpha_j,\beta_j,\gamma_j,\delta_j}$. Hence, for every quaternary lattice term $f$, we have that 
 $f(\alpha_i,\beta_i,\gamma_i,\delta_i)=f(\alpha_j,\beta_j,\gamma_j,\delta_j)$. This implies that for every 
$\eta=\tuple{\eta_1,\dots,\eta_t}\in \sublat{\alpha,\beta,\gamma,\delta}$, we have that $\eta_i=\eta_j$. Thus, $\sublat{\alpha,\beta,\gamma,\delta}\neq \Part n^t$, which is a contradiction proving  Remark~\ref{remnVrbG}.
\end{proof}

\begin{remark}\label{remngyznnGynrGcrsT} For a four-generated finite lattice $L$, the direct square $L^2$ of $L$ need not be four-generated. For example, if $L$ is the distributive lattice generated freely by four elements, then there exists no $t\geq 2$ such that $L^t$ is four-generated.
\end{remark}

\begin{proof}
Let $t\geq 2$, and let $L$ be the free distributive lattice on four generators. Observe that $L^t$ is distributive. So if $L^t$ was four-generated, then it would be a homomorphic image of $L$ and  $|L|^t=|L^t|\leq |L|$ would be a contradiction.
\end{proof}

\begin{proof}[Proof of Theorem~\ref{thmmain}] 
Since the notation of the elements of the base set is irrelevant, it suffices to show that $\Equ{Z(n)}^m$ is four-generated. No matter if $n$ is odd or even, we 
use the notation $k$, $a_i$ and $b_j$ as in Figures~\ref{figd2} and \ref{figd3}. We are going to define $\valpha=\tuple{\alpha_1,\dots, \alpha_m}$, $\vbeta=\tuple{\beta_1,\dots, \beta_m}$, $\vgamma=\tuple{\gamma_1,\dots, \gamma_m}$, and $\vshd=\tuple{\shd_1,\dots, \shd_m}$ so that $\set{\valpha,\vbeta,\vgamma,\vshd}$ generates $\Equ{Z(n)}^m$. For every $i\in\set{1,\dots,m}$,  $\alpha_i$, $\beta_i$, and $\gamma_i$ are defined as in Figures~\ref{figd2} and \ref{figd3}, that is, as in the proofs of Lemmas~\ref{lemmazadori} and \ref{lemmazeven}. 
However, the definition of the equivalences $\shd_i$ is going to be more tricky.
Let $\delta_i:=\equ{a_0}{b_0}+\equ{a_k}{b_{k-1}}$, as in Lemmas~\ref{lemmazadori} and \ref{lemmazeven}. Note that 
\begin{equation}
\text{none of $\alpha_i:=\alpha$, $\beta_i:=\beta$, $\gamma_i:=\gamma$, and $\delta_i:=\delta$ depends on $i$.}
\label{eqtxtngkSmfmRstsRz}
\end{equation}
We know from Lemmas~\ref{lemmazadori} and \ref{lemmazeven} that $\set{\alpha_i,\beta_i,\gamma_i,\delta_i}$ generates $\Equ{Z(n)}$. Therefore, for any two distinct elements $u$ and $v$ of $Z(n)$, we can pick a quaternary lattice term $\fterm u v= \fterm u v(\oal,\obe,\oga,\ode)$ with variables $\obmu:=\tuple{\oal,\obe,\oga,\ode}$ such that, in virtue of \eqref{eqczhhndPMtlvRdwdb} and \eqref{eqnmsdsulTVcX},
\begin{equation}\left.
\parbox{8.7cm}{depending on the parity of $n$, $\fterm u v$ is $\eterm u v$ from the proof of Lemma~\ref{lemmazadori} or it is 
 $\fterm u v^\ast$ from that of Lemma~\ref{lemmazeven}, and 
$\fterm u v(\alpha_i,\beta_i,\gamma_i,\delta_i)=\equ u v \in \Equ{Z(n)}$.}\,\,\right\}
\label{eqfTmrzGb}
\end{equation}
By defining $\fterm u u$ to be the meet of its four variables, the validity of\eqref{eqfTmrzGb} extends to the case $u=v$, where $\equ u u$ is understood as the least partition, that is, the partition with all of its blocks being singletons.

Next, let 
\begin{equation}
U:=\set{a_1,a_2\dots, a_{k-1}}\,\text{ and }\,W:=\set{b_0,b_1,\dots,b_{k-1}};
\label{eqwmghUV}
\end{equation}
these sets are indicated by dashed ovals in Figures~\ref{figd2}, \ref{figd3}, and \ref{figd4}. 
By the definition of $\maxs(k-1)$, we can pick an integer
$r'\in\NN$ such that there are exactly $\maxs(k-1)$ equivalences of $U$ with exactly $r'$ blocks. (By a block of an equivalence we mean a block of the corresponding partition.)  Let $\aGa$ denote the set of these  ``$r'$-block equivalences'' of $U$. Clearly, $\aGa$ is an antichain in $\Equ U$ with size $|\aGa|=\maxs(k-1)$. 
Similarly, $\maxs(k)$ is the number of $r''$-block equivalences  for some $r''\in\NN$ and the $r''$-block equivalences of $W$ form an antichain $\aHa\subseteq \Equ W$ such that $|\aHa|=\maxs(k)$. 
Observe that, in the direct product $\Equ U\times\Equ W$,
\begin{equation}
\aGa\times \aHa \text{ is an antichain}.
\label{eqanTiChaIn}
\end{equation}
Since $|\aGa\times \aHa|=|\aGa|\cdot|\aHa|=\maxs(k-1)\cdot \maxs(k)=m$, see \eqref{eqmSrpRdgB},
we can enumerate $\aGa\times \aHa$ in the following repetition-free list of length $m$ as follows:
\begin{equation}
\aGa\times \aHa=\set{\pair{\kappa_1}{\lambda_1}, \pair{\kappa_2}{\lambda_2}, \dots, \pair{\kappa_m}{\lambda_m}  }.
\label{eqchzTnFshRp}
\end{equation}
For each $i\in\set{1,\dots,m}$, we define $\shd_i$ as follows:
\begin{equation}
\shd_i:=  \text{the equivalence generated by }\delta_i \cup \kappa_i \cup \lambda_i;
\label{eqctznBkPrDTrMp}
\end{equation}
this makes sense since each of $\delta_i$, $\kappa_i$ and $\lambda_i$ is a subset of $Z(n)\times Z(n)$.
Clearly, for any $x\neq y\in Z(n)$,   $\pair x y\in \alpha\shd_i$ if and only $\pair x y\in\kappa_i\cup\lambda_i\subseteq U^2\cup W^2$. This fact together with $\kappa_i\cap\lambda_i\subseteq U^2\cap W^2=\emptyset$ and \eqref{eqanTiChaIn}  yield that 
for any $i,j\in\set{1,2,\dots,m}$,
\begin{equation}
\text{if $i\neq j$, then $\alpha \shd_i$ and $\alpha \shd_j$ are incomparable.}
\label{pbxTfhhsslqkntLXn}
\end{equation}

\begin{figure}[htb] 
\centerline
{\includegraphics[scale=1.0]{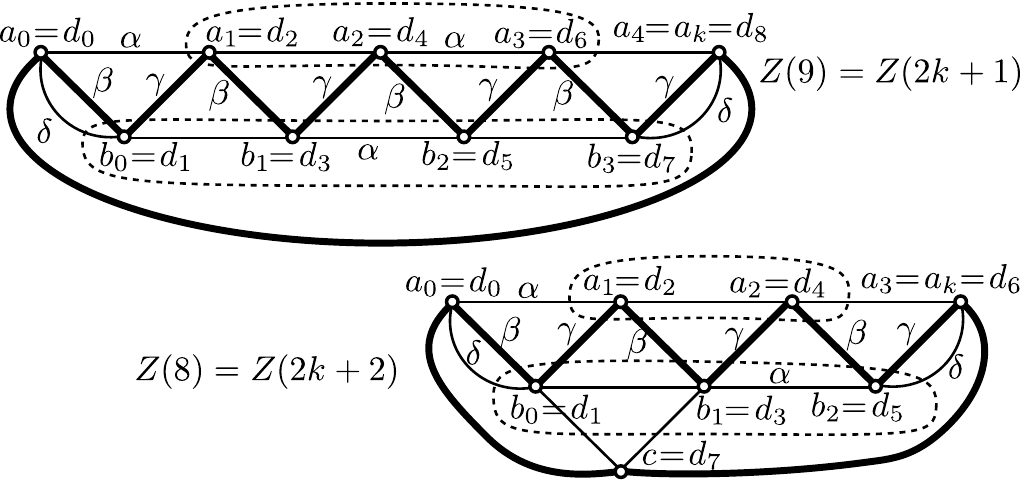}}
\caption{``zigzagged circles''
\label{figd4}}
\end{figure}%

Next, we define 
\begin{equation}
\text{the ``\emph{zigzagged circle}'' }\,\,
\tuple{d_0,d_1,\dots, d_{n-1}}
\label{eqtxtZgZghsvTjs}
\end{equation}
as follows; see also the thick edges and curves in Figure~\ref{figd4}. 
(Note that the earlier meaning of the notation $d_0, d_1,\dots$ is no longer valid.)
For $i\in\set{0,\dots,k-1}$, we let $d_{2i}:=a_i$ and $d_{2i+1}=b_i$. We let $d_{2k}=a_k$ and, if $n=2k+2$ is even, then we let $d_{n-1}=c$. Two consecutive vertices of the zigzagged circle
will always be denoted by $d_p$ and $d_{p+1}$ where $p, p+1\in\set{0,1,\dots,n-1}$ and the addition is understood modulo $n$. The zigzagged circle has one or two \emph{thick curved edges}; they are $\pair{a_0}{a_k}$ for $n=2k+1$ odd and they are  $\pair{a_0}c$ and $\pair{c}{a_k}$ for $n=2k+2$ even; the rest of its edges are \emph{straight thick edges}. So the zigzagged circle consist of the thick (straight and curved) edges, whereby the adjective ``thick'' will often be dropped.

Next, we define some lattice terms associated with the edges of the zigzagged circle. Namely,
for  $j\in\set{1,\dots,m}$ and for 
$p\in \set{0,1,\dots,2k-1}$, we define the quaternary term
\begin{equation}
\begin{aligned}
\gterm j {d_p} {d_{p+1}}(\obmu):=
\fterm {d_p} {d_{p+1}}(\obmu)
&\cdot \prod_{\pair {d_p}x \in \alpha\shd_j} \bigl(  \oal\ode +  \fterm x{d_{p+1}}(\obmu)  \bigr) 
\cr
&\cdot
\prod_{\pair y{d_{p+1}}\in \alpha\shd_j} \bigl(\fterm {d_{p}}y(\obmu) + \oal\ode \bigr).
\end{aligned}
\label{eqcBmTzXQfFt}
\end{equation}
The assumption on $p$ means that \eqref{eqcBmTzXQfFt} defines $\gterm j {d_p} {d_{p+1}}(\obmu)$ for each straight edge of the zigzagged circle. 
We claim that for all $j\in\set{1,\dots,m}$ and $p\in \set{0,1,\dots,2k-1}$, 
\begin{equation}
\gterm j {d_p} {d_{p+1}}(\alpha_j,\beta_j,\gamma_j,\shd_j)=
\equ {d_p} {d_{p+1}}.
\label{eqCshTnGbW}
\end{equation}
In order to show \eqref{eqCshTnGbW}, observe that 
$\beta_i\shd_i=\beta_i\delta_i=\equ{a_0}{b_0}$ and 
$\gamma_i\shd_i=\gamma_i\delta_i=\equ{a_k}{b_{k-1}}$. These equalities, \eqref{eqpbxZbntGhSwD},  \eqref{eqnZhgjdlSrkhllRw}, and \eqref{eqfTmrzGb} yield that for any $u,v\in Z(n)$, $i\in\set{1,\dots,m}$, and $p\in\set{0,1,\dots,2k-1}$, 
\begin{align}
\fterm {u} {v}(\alpha_i,\beta_i,\gamma_i,\shd_i)&=\equ {u} {v}\,\,\text{ and, in particular,}  \label{eqalignZrtbvRsrst}
\\
\fterm {d_p} {d_{p+1}}(\alpha_i,\beta_i,\gamma_i,\shd_i)&=\equ {d_p} {d_{p+1}}.
\label{eqcHtRbkvWp}
\end{align}
Combining \eqref{eqcBmTzXQfFt} and  \eqref{eqcHtRbkvWp}, 
we obtain the ``$\leq$'' part of \eqref{eqCshTnGbW}. In order to turn this inequality to an equality, we have to show that the pair  
$\pair{d_p}{d_{p+1}}$ belongs to $\alpha\shd_j+\fterm x{d_{p+1}} (\alpha_j,\beta_j,\gamma_j,\shd_j ) $ for every $\pair {d_p}x \in \alpha\shd_j$, and it also belongs to 
$\fterm {d_{p}}y(\alpha_j,\beta_j,\gamma_j,\shd_j) + \alpha\shd_j$ for every $\pair y{d_{p+1}}\in \alpha\shd_j$.
But this is trivial since $\pair x{d_{p+1}}\in \fterm x{d_{p+1}} (\alpha_j,\beta_j,\gamma_j,\shd_j )$ in the first case by \eqref{eqalignZrtbvRsrst}, and similarly trivial in the second case. 
We have shown  \eqref{eqCshTnGbW}.

Next, we claim that for any $i,j\in\set{1,\dots,m}$,
\begin{equation}\left.
\parbox{8cm}{if $i\neq j$, then there exists a $p\in
\set{0,1,\dots, 2k-1}$ such that 
$\gterm j {d_p} {d_{p+1}}(\alpha_i,\beta_i,\gamma_i,\shd_i)=\enul:=0_{\Equ{Z(n)}}$.
}\,\,\right\}
\label{eqpbxZhgRblKnGgvnD}
\end{equation}
In order to prove \eqref{eqpbxZhgRblKnGgvnD}, assume that $i\neq j$. For an equivalence $\epsilon\in \Equ{Z(n)}$ and $x\in Z(n)$, the \emph{$\epsilon$-block} $\set{y\in Z(n): \pair x y\in \epsilon}$ of $x$ will be denoted by $x/\epsilon$. 
We know from \eqref{pbxTfhhsslqkntLXn} that $\alpha\shd_j\not\leq\alpha\shd_i$. Hence, there is an element $x\in Z(n)$ such that $x/(\alpha\shd_j)\not\subseteq  x/(\alpha\shd_i)$. Since  $c/(\alpha\shd_j)=\set c= c/(\alpha\shd_i)$ for $n$ even, $x$ is distinct from $c$. Hence, $x$ is one of the endpoints of a straight edge $\pair{d_p}{d_{p+1}}$ of the zigzagged circle. This is how we can select a $p\in\set{0,1,\dots, 2k-1}$, that is, 
a straight edge $\pair{d_p}{d_{p+1}}$ of the zigzagged circle \eqref{eqtxtZgZghsvTjs}  such that 
\begin{equation}
d_p/(\alpha\shd_j)\not\subseteq  d_p/(\alpha\shd_i) \quad \text{ or }\quad 
d_{p+1}/(\alpha\shd_j) \not\subseteq  d_{p+1}/(\alpha\shd_i).
\label{eqchSdjgThRgRTrnL}
\end{equation} 
Now, we are going to show that this $p$ satisfies the requirement of \eqref{eqpbxZhgRblKnGgvnD}. We can assume that the first part of the disjunction given in \eqref{eqchSdjgThRgRTrnL} holds, because the treatment for the second half is very similar. Pick an element 
\begin{equation}
z\in d_p/(\alpha\shd_j)\text{ such that }z\notin  d_p/(\alpha\shd_i). 
\label{eqChClmTnsz}
\end{equation} 
Because of \eqref{eqcHtRbkvWp} and the first meetand in \eqref{eqcBmTzXQfFt}, 
\begin{equation}
\gterm j {d_p} {d_{p+1}}(\alpha_i,\beta_i,\gamma_i,\shd_i)\leq \equ {d_p} {d_{p+1}}.
\label{eqnMsTsnWRStbHws}
\end{equation}
We claim that 
\begin{equation}
\pair{d_p} {d_{p+1}}\notin \gterm j {d_p} {d_{p+1}}(\alpha_i,\beta_i,\gamma_i,\shd_i).
\label{eqhzSzNtlNsrLlgz}
\end{equation}
Suppose the contrary. Then, using \eqref{eqcBmTzXQfFt} and that
$\pair {d_p}z\in\alpha\shd_j$ by \eqref{eqChClmTnsz}, we have that 
\begin{equation}
\pair{d_p} {d_{p+1}} \in \alpha\shd_i + \fterm z{d_{p+1}}(\alpha_i,\beta_i,\gamma_i,\shd_i)
\eeqref{eqalignZrtbvRsrst}
\alpha\shd_i + \equ z{d_{p+1}}. 
\label{eqghzrBsjTvnScsz}
\end{equation}
According to \eqref{eqghzrBsjTvnScsz}, there exists a \emph{shortest} sequence $u_0=d_{p+1}$, $u_1$, \dots, $u_{q-1}$, $u_q=d_p$ such that for every $\ell\in\set{0,1,\dots, q-1}$, 
either $\pair{u_\ell}{u_{\ell+1}}\in \alpha\shd_i$, which is called a \emph{horizontal step}, or  $\pair{u_\ell}{u_{\ell+1}}\in \equ z{d_{p+1}}$, which is a \emph{non-horizontal step}.
There is at least one non-horizontal steps since $d_p$ and $d_{p+1}$ are in distinct $\alpha$-blocks.
A non-horizontal step means that $\set{u_\ell,u_{\ell+1}} =\set{z,d_{p+1}}$, so $\set{z,d_{p+1}}$ is the only ``passageway'' between the two nonsingleton $\alpha$-blocks.  Hence, there exists exactly one non-horizontal step since our sequence is repetition-free. This step is the first step since we have taken a shortest sequence. Hence, 
$u_1=z$ and all the subsequent steps are horizontal steps. 
Hence, $\pair {z}{d_p}=\pair {u_1}{d_p}\in\alpha\shd_i$. 
Thus, $z\in d_p/(\alpha\shd_i)$, contradicting the choice of $z$ in \eqref{eqChClmTnsz}. This contradiction yields \eqref{eqhzSzNtlNsrLlgz}. Finally,  \eqref{eqhzSzNtlNsrLlgz} together with \eqref{eqnMsTsnWRStbHws} imply \eqref{eqpbxZhgRblKnGgvnD}.

Next, for $j\in\set{1,2,...,m}$ and $q\in\set{0,1,\dots, n-1}$, we define the following quaternary term
\begin{equation}\left.
\begin{aligned}
\hterm j {d_q}{d_{q+1}}(\obmu):= 
\cr\fterm {d_q}{d_{q+1}}(\obmu)
&\cdot
\prod_{p=0}^{2k-1}\bigl(\fterm {d_q}{d_p}(\obmu) + \gterm j{d_p}{d_{p+1}}(\obmu) + \fterm {d_{p+1}}{d_{q+1}}(\obmu)\bigr)\cr
&\cdot
\prod_{p=0}^{2k-1}\bigl(\fterm {d_q}{d_{p+1}}(\obmu) + \gterm j{d_p}{d_{p+1}}(\obmu) + \fterm {d_{p}}{d_{q+1}}(\obmu)\bigr),
\end{aligned}
\,\,\,\right\}
\label{eqttzsnJsjPl}
\end{equation}
where $q+1$ in subscript position is understood modulo $n$.
We claim that, for $q\in\set{0,1,\dots, n-1}$ and $i,j\in\set{1,\dots, m}$, 
\begin{equation}
\hterm j  {d_q}{d_{q+1}}(\alpha_i,\beta_i,\gamma_i,\shd_i)=
\begin{cases}
\equ  {d_q}{d_{q+1}},&\text{if }\,\, i=j,\cr
\enul=0_{\Equ{Z(n)}},&\text{if }\,\, i\neq j.
\end{cases}
\label{eqHgtqrmBnSkWrkK}
\end{equation}
In virtue of \eqref{eqCshTnGbW},  \eqref{eqalignZrtbvRsrst}, and \eqref{eqcHtRbkvWp}, the validity of \eqref{eqHgtqrmBnSkWrkK} is clear when $i=j$. So, to prove  \eqref{eqHgtqrmBnSkWrkK}, we can assume that $i\neq j$. 
Since 
$\hterm j  {d_q}{d_{q+1}}(\alpha_i,\beta_i,\gamma_i,\shd_i)\leq \equ  {d_q}{d_{q+1}}$ by \eqref{eqalignZrtbvRsrst} and \eqref{eqcHtRbkvWp}, it suffices to show that 
$\pair{d_q}{d_{q+1}} \notin \hterm j  {d_q}{d_{q+1}}(\alpha_i,\beta_i,\gamma_i,\shd_i)$. Suppose the contrary. Then we obtain from  \eqref{eqalignZrtbvRsrst} and  \eqref{eqttzsnJsjPl} that for 
all $p\in\set{0,\dots,2k-1}$,
\begin{align}
\pair{d_q}{d_{q+1}} &\in \equ{d_q}{d_p} + \gterm j{d_p}{d_{p+1}}(\alpha_i,\beta_i,\gamma_i,\shd_i)
+ 
\equ{d_{p+1}}{d_{q+1}}\text{ and}
\label{eqlgnbmTzskWrNztp}\\
\pair{d_q}{d_{q+1}} &\in \equ{d_q}{d_{p+1}} + \gterm j{d_p}{d_{p+1}}(\alpha_i,\beta_i,\gamma_i,\shd_i) + \equ {d_{p}}{d_{q+1}}.
\label{fzhNtbmTzrTxT}
\end{align}
Now we choose $p$ according to \eqref{eqpbxZhgRblKnGgvnD}; then $\gterm j{d_p}{d_{p+1}}(\alpha_i,\beta_i,\gamma_i,\shd_i)$ can be omitted from  \eqref{eqlgnbmTzskWrNztp} and \eqref{fzhNtbmTzrTxT}. Therefore, if $p=q$, then \eqref{eqlgnbmTzskWrNztp}  asserts that $\pair{d_q}{d_{q+1}}\in\enul$, a contradiction. Note that, due to $n\geq 5$, $p=q$ is equivalent to  $|\set{d_p,d_{p+1},d_q,d_{q+1}}|=2$, whence $|\set{d_p,d_{p+1},d_q,d_{q+1}}|=2$ has just been excluded. 
If $|\set{d_p,d_{p+1},d_q,d_{q+1}}|=4$, then each of   \eqref{eqlgnbmTzskWrNztp} and \eqref{fzhNtbmTzrTxT} gives a contradiction again. If $|\set{d_p,d_{p+1},d_q,d_{q+1}}|=3$, then exactly one of  \eqref{eqlgnbmTzskWrNztp} and \eqref{fzhNtbmTzrTxT} gives a contradiction. Hence, no matter how $p$ and $q$ are related, we obtain a contradiction. This proves the $i\neq j$ part of \eqref{eqHgtqrmBnSkWrkK}.  Thus,  \eqref{eqHgtqrmBnSkWrkK} has been proved.

Finally, let $K:=\sublat{\valpha,\vbeta,\vgamma,\vshd}$; it is a sublattice of $\Equ{Z(n)}^m$ and we are going to show that 
$K=\Equ{Z(n)}^m$. Let $j\in\set{1,\dots,m}$. 
It follows from \eqref{eqHgtqrmBnSkWrkK} that
\begin{equation}
\tuple{
\enul,\dots,\enul,
\underbrace{\equ {d_q} {d_{q+1}}}
  _{ j\text{-th entry}  }
,\enul,\dots,\enul  }
\in K,\text{ for all }q\in\set{0,\dots,n-1}.
\label{eqmsTlgYkZsWsG}
\end{equation}
Since the sublattice 
\begin{equation*}
S_j:=\set{\enul}\times\dots\times \set{\enul}\times\Equ{Z(n)}\times \set{\enul}\times\dots\times \set{\enul}
\end{equation*}
with the non-singleton factor at the $j$-th place
is isomorphic to $Z(n)$, it follows from \eqref{eqmsTlgYkZsWsG} and Lemma~\ref{lemmaHamilt} that $S_j\subseteq K$, for all $j\in\set{1,\dots,m}$. Therefore, since every element of $\Equ{Z(n)}^m$ is of the form $s^{(1)}+ s^{(2)} + \dots + s^{(m)}$ with $s^{(1)}\in S_1$, \dots, $s^{(m)}\in S_m$, we obtain that $\Equ{Z(n)}^m \subseteq K$. Consequently, 
$\Equ{Z(n)}^m =K = \sublat{\alpha,\beta,\gamma,\shd}$ is a four-generated lattice, as required.
The proof of Theorem~\ref{thmmain} is complete.
\end{proof}

\section{$(1+1+2)$-generation}\label{sectootwo}
By a $(1+1+2)$-generating set or, in other words, \emph{a generating subset of order type $1+1+2$}
we mean a four element generating set such that exactly two of the four elements are comparable. 
Lattices having such a generating set are called \emph{$(1+1+2)$-generated.}
In his paper, Z\'adori~\cite{zadori} proved that for every integer $n\geq 7$, the partition lattice $\Part n$ is $(1+1+2)$-generated. In this way, he improved the result proved by Strietz~\cite{strietz2} from $\set{n: n\geq 10}$ to $\set{n: n\geq 7}$. In this section, we generalize this result to direct powers by the following theorem;
\eqref{pbxPartnmM} and \eqref{eqtxtmMxSdfLm} are still in effect.

\begin{theorem}\label{thmoot} Let $n\geq 7$ be an integer, let 
$k:=\lfloor (n-1)/2 \rfloor$, and let 
\begin{equation}
\csm=\csm(n):=
\max\left(\,\, \lfloor (k-1)/2 \rfloor,\,\,   \maxs(\lfloor (k-1)/2 \rfloor)^2 \,\,   \right).
\label{eqmSrmd nmrpRd}
\end{equation}
Then $\Part n^\csm$ or, equivalently, $\Equ n^\csm$ is $(1+1+2)$-generated. In other words, the $\csm$-th direct power of the lattice of all partitions of the set $\set{1,2,\dots, n}$ is has a generating subset of order type $1+1+2$.
\end{theorem}

Note that $\csm$ above is at least $1$, and $\csm \geq 2$ if and only if $n\geq 11$. 
\begin{figure}[htb] 
\centerline
{\includegraphics[scale=1.0]{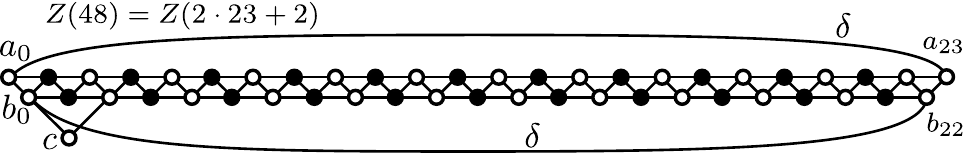}}
\caption{$Z(2k+2)$ for $k=23$
\label{figd5}}
\end{figure}%

\begin{proof} 
With our earlier conventions, we define $\alpha$, $\beta$, and $\gamma$ as in Sections~\ref{sectztrms} and \ref{sectprod}, see also \eqref{eqtxtngkSmfmRstsRz},  but we let $\delta:=\equ{a_0}{a_k}+\equ{b_0}{b_{k-1}}\in\Equ{Z(n)}$. For $n=47$, this is illustrated by Figure~\ref{figd5} if we omit vertex $c$. For $n=48$, Figure~\ref{figd5} is a faithful illustration without omitting anything but taking \eqref{eqtxtsdhClrsznzlD} into account. Instead of working with
$U$ and $W$ from \eqref{eqwmghUV}, we define these two sets as follows.
\begin{align}
U&:=\set{a_i: 1\leq i\leq k-2\text{ and }i\text{ is odd}}\text{ and}\label{eqhrTdfzgu}
\\ 
W&:=\set{b_i: 1\leq i\leq k-2\text{ and }i\text{ is odd}}.
\label{eqhrTdfzgv}
\end{align}
In Figure~\ref{figd5}, $U\cup W$ is the set of black-filled elements. Let $r:=\lfloor (k-1)/2 \rfloor$; note that $r=|U|=|W|$. Since $n\geq 7$, we have that $k\geq 3$ and $r\geq 1$.
Let $\bdelta$ be the equivalence on $Z(n)$ generated by
$\delta\cup U^2\cup W^2$. In other words, $\bdelta$ is the equivalence with blocks  $\set{a_0,a_k,b_0,b_{k-1}}$, $U$, and $W$such that the rest of its blocks are singletons. 
Let 
\[t:=2\cdot\lfloor(k+1)/2\rfloor -3\,\,\text{ and }\,\,T:=\set{i: 1\leq i\leq t \text{ and }i\text{ is odd} }.
\]
Based on Figure~\ref{figd5}, we can think of $T$ as the set of subscripts of the black-filled elements.
The blocks of  $\gamma+\bdelta$ are the following:  
\begin{align*}
&\set{a_0,a_k,b_0,b_{k-1}}\cup \set{a_i: i\in T}\cup\set{b_{i-1}:i\in T},\cr
&\set{a_{i+1}:i\in T}\cup\set{b_{i}:i\in T},\text{ and, if }k\text{ is even, }\set{a_{k-1},b_{k-2}},
\end{align*}
and, for $n$ even, $\set{c}$.
Hence, we obtain that
\begin{equation}
\equ{a_0}{b_0} =\beta(\gamma+\bdelta).
\label{eqrzhGrnszGsFpL}
\end{equation}
Similarly, the blocks of  $\beta+\bdelta$ are $\set{c}$, if $n$ is even, and the following:
\begin{align*}
&\set{a_{i}:i\in T}\cup\set{b_{i}:i\in T},\,\,\,\set{a_0,a_k,b_0,b_{k-1}, a_{k-1}}.\cr
&\text{and, for }i\in\set{2,3,\dots,k-2}\setminus T,\,\,\,\,\set{a_i,b_i}.
\end{align*}
Hence, it follows that 
\begin{equation}
\equ{a_k}{b_{k-1}} =\gamma(\beta+\bdelta).
\label{eqnhsptdkjrTSz}
\end{equation}
In the proof of Theorem~\ref{thmmain}, based on  \eqref{eqwmghUV},  $\aGa$, $\aHa$, \eqref{eqanTiChaIn}, and \eqref{eqchzTnFshRp}, we defined the equivalences $\shd_1$,\dots,$\shd_m$ in \eqref{eqctznBkPrDTrMp}. Now we define  $\shd_1$,\dots,$\shd_\csm$ exactly in the same way but we use \eqref{eqhrTdfzgu} and \eqref{eqhrTdfzgv} instead of  \eqref{eqwmghUV}, and we take into account that  $U$ and $W$ are now smaller and  we obtain $\csm$ rather than $m$ from them. Observe that
\begin{equation}
\equ{a_0}{b_0} =\beta(\gamma+\delta)\quad\text{ and }\quad
\equ{a_k}{b_{k-1}} = \gamma(\beta+\delta).
\label{eqezrhBsmjRtMcLc}
\end{equation}
Since  $\delta\leq \shd_j\leq \bdelta$ for  $j\in\set{1,\dots,\csm}$, it follows from \eqref{eqrzhGrnszGsFpL}, \eqref{eqnhsptdkjrTSz}, and \eqref{eqezrhBsmjRtMcLc} that
\begin{equation}
\equ{a_0}{b_0} =\beta(\gamma+\shd_j)\quad\text{ and }\quad
\equ{a_k}{b_{k-1}} = \gamma(\beta+\shd_j)
\label{eqwzmbHhGRdffkT}
\end{equation}
for all $j\in\set{1,\dots,\csm}$. 
Armed \eqref{eqwzmbHhGRdffkT}
and all the previous preparations, the rest of the proof is the same as in case of Theorem~\ref{thmmain} unless $r=2$;  these details are not repeated here. 
Observe that the only role of $m$ in the proof of Theorem~\ref{thmmain} is that we had to find an $m$-element antichain in 
$\Equ U\times \Equ W$. Similarly, if $r\neq 2$, then all what we have to do with $\csm$ is to find an $\csm$-element antichain in $\Equ U\times \Equ W$. If $r\neq 2$, then we obtain such an antichain as the Cartesian product of an antichain of $\Equ U$ and that of $\Equ W$. If $r=2$, then this method does not work since $\Equ U$ and $\Equ W$ are (two-element) chains but $\csm=2$. However, $\Equ U\times \Equ W$ has an  $\csm=2$-element antichain even in this case. This completes the proof of Theorem~\ref{thmoot}.
\end{proof}

Obviously, Theorem~\ref{thmoot} implies  the following  counterpart of Corollary~\ref{coroLprd}.

\begin{corollary}\label{corodjTd}
Let $n$ and $\csm$ be as in Theorem \ref{thmoot}. Then for every integer $t$ with $1\leq t\leq \csm$, the direct power $\Part n^t$ is $(1+1+2)$-generated.
\end{corollary}

\section{Authentication and secret key cryptography with  lattices}\label{sectauth}

While lattice theory is rich with involved constructs and proofs, it seems not to have many, if any, applications in  information theory. The purpose of this section is to suggest a protocol primarily for \emph{authentication}; it is also good for \emph{secret key cryptography}, and it could be appropriate for a \emph{commitment protocol}.

Assume that during the authentication protocol that we are going to outline, \aph{}\footnote{\emph{\aph} is the Hungarian version of Andrew;
as a famous lattice theorist with this first name, I mention my  scientific advisor, \aph{} P.\ Huhn (1947--1985).} intends to prove his identity to his Bank and conversely;  online, of course. In order to do so, \aph{} and the Bank should find  a lattice $L$ with the following properties: 
\begin{itemize}
\item $|L|$ is large, 
\item $L$ has a complicated structure, 
\item the length of $L$ is small (that is, all maximal chains of  $L$ are small),
\item every non-zero element of $L$ has lots of lover covers and dually,
\item $L$ can be given by and constructed easily from little data, 
\item and $L$ is generated by few elements. 
\end{itemize}
The first four properties are to make the Adversary's task difficult (and practically impossible) while the rest of these properties ensure that \aph{} and the Bank can handle $L$. It is not necessary that $L$ has anything to do with partitions, but partitions lattices and their direct powers seem to be good choices. Partition lattices are quite complicated since every finite lattice can be embedded into a finite partition lattice by Pudl\'ak and T\r uma~\cite{pudlaktuma}. Also, they are large lattices described by very little data. For example, we can take
\begin{align}
L&=\Part{273},\text{ its size is }|\Part{273}|\approx 3.35\cdot 10^{404},\text{ or}\label{eqZhgRtBfktszg} 
\\
L&=\Part{12}^{61},\text{ its size is }|\Part{12}^{61}|\approx 1.27\cdot 10^{404}.
\label{eqdmzmgrYjHBbsgnwsztRtnk}
\end{align}
Although these two lattices seem to be similar in several aspects, each of them has some advantage over the other.
As opposed to  $\Part{273}$,
\begin{equation}
\parbox{7.7cm}{joins can easily be computed componentwise in  $\Part{12}^{61}$ if parallel computation is allowed.}
\end{equation}  
On the other hand, using that $\Part{273}$  is a semimodular lattice and so any two of its maximal chains have the same length, it is easy to see that the longest chain in $\Part{273}$
is only of length 272 (that is, this chain consists of 273 element). Using semimodularity again, it follows easily  that the longest chain in $\Part{12}^{61}$ is of length $61\cdot 11=671$, so  $\Part{273}$ seems to be more advantageous in this aspect.  Based on  data obtained by computer, to be presented in tables \eqref{tableD8gens} and \eqref{tablenFpgFns}, we guess that $\Part{273}$ has more $p$-element generating sets of an \emph{unknown pattern} than $\Part{12}^{61}$. If so, then this can also be an advantage of  $\Part{273}$ 
since a greater  variety of $p$-element generating sets of unknown patterns makes the Adversary's task even more hopeless. It is probably too early to weigh  all the pros and cons of \eqref{eqZhgRtBfktszg}, \eqref{eqdmzmgrYjHBbsgnwsztRtnk} and, say, $\Part{113}^{3}$ with size $|\Part{113}|^{3}\approx 1.51\cdot 10^{405}$.

\aph{} and the Bank choose two small integer parameters $p,q\geq 4$, the suggested value is $p=q=8$ or larger; these numbers can be public. 
Also, \aph{} and the Bank agree upon a $p$-tuple
\begin{equation}
\ves=\tuple{s_1,s_2,\dots,s_{p}} \in L^{p}.
\label{eqphndThbnhnSpLSr}
\end{equation} 
This $\vec s$ is the common authentication code for \aph{} and the Bank; only they know it and they keep it in secret. So far, the role of $\ves$ is that of the PIN (personal identification number) of a bank card.

Every time \aph{} intends to send an authenticated message to the Bank,  the Bank selects a vector $\vew=\tuple{w_1,w_2,\dots, w_q}$ of long and complicated $p$-ary lattice terms
randomly. (We are going to discuss after \eqref{eqmnTltjrm} how to select $\vew$.)
Then the Bank sends $\vew$ to \aph{}. 
(If \aph{} thinks that  $\vew$ is not complicated enough, then he is allowed to ask for a more complicated $\vew$ repeatedly until he is satisfied with $\vew$.) Then, to prove his identity, \aph{} sends 
\begin{equation}
\vew(\ves):=\tuple{w_1(s_1,\dots,s_p),\dots, w_q(s_1,\dots, s_p)  }
\label{eqszhnKhKypNxGqr}
\end{equation}
to the Bank. (Preferably, in the same message that instructs the Bank to do something like transferring money, etc.) The Bank also computes $\vew(\ves)$ and compares it with what \aph{} has sent; if they are equal then the Bank can be sure that he communicates with \aph{} rather than with an adversary. Note that it is easy and fast to compute $\vew(\ves)$ from $\vew$ and $\ves$.  
Note also that, changing their roles, \aph{} can also verify (by another $q$-tuple $\vew'$ of terms) that he communicates with the Bank rather than with the Adversary.

The point of the protocol is that while $\ves$ can be used many times, 
a new $\vew$ is chosen at each occasion. So even if the Adversary intercepts the communication, he cannot use the old values of $\vew(\ves)$. So the Adversary's only chance to interfere is to extract
the secret $\ves$ from $\ver:=\vew(\ves)$. However, extracting $\ves$ from $\ver:=\vew(\ves)$ and $\vew$ seems to be hard. (This problem is in NP and hopefully it is not in $P$.)  The Adversary cannot test all possible $p$-tuples $\ves'\in L^p$ since there are astronomically many such tuples. The usual iteration technique to find a root of a function $\mathbb R^p\to \mathbb R$ is not applicable here since, in general, 
\begin{equation}
\text{it is unlikely that two elements of $L$ are comparable,}
\label{eqtxtnmctKzSszbwhnT} 
\end{equation}
simply because the length of $L$ is small but $|L|$ is large.  It is also unlikely that two members of $L^q$ are comparable. 
If the Adversary begins parsing, say, $r_1:=w_1(\ves)$, then even the first step splits into several directions since $r_1\in L$ has many lower and upper covers and so there are many possibilities to represent it as the join of two elements (in case the outmost operation sign in $w_1$ is $\vee$) or as the meet of two elements (in case the outmost operation sign is $\wedge$).
 Each of these several possibilities split into several cases at the next step, and this happens many times depending on the length of $w_1$. But $w_1$ is a long term, whence exponentially many sub-directions should be handled, which is not feasible.

Some caution  is necessary when choosing the common secret authentication code $\ves$. This $\ves$ should be chosen so that $\sublat{s_1,s_2,\dots, s_p}=L$ or at least $\sublat{s_1,s_2,\dots, s_p}$ should be very large. One possibility to ensure that $\set{s_1,s_2,\dots, s_p}$ generates $L$ is to extend a four-element generating set from Sections~\ref{sectztrms}--\ref{sectootwo} to a $p$-element subset of $L$. 
If $L=\Part{273}$, then one can pick a permutation $\tau$ of the set $\set{1,2,\dots,273}$; this $\tau$ induces an automorphism $\otau$ of $\Part{273}$ in the natural way, and 
$\set{\otau(\alpha),\otau(\beta),\otau(\gamma),\otau(\delta)}$
 with $\alpha,\dots, \delta$ from  Section~\ref{sectztrms} is 
a four-element generating set of $\Part{273}$.
If $L=\Part{12}^{61}$, then in addition to the permutations of $\set{1,2,\dots,12}$, allowing different permutations in the direct factors, 
there are many ways to select a 61-element antichain as a subset of the 175-element maximum-sized antichain that occurs in \eqref{eqanTiChaIn}. (Note that we obtained this number, 175, when computing the last column of \eqref{tablerdDbsa}.) In both cases, \aph{} and the Bank can easily pick one of the astronomically many four-element generating sets described in the present paper. A four-element generating set can be extended  to a $p$-element one in many ways. It would be even better to pick a $p$-element generating set of an unknown pattern, but it is not clear at this moment how this would be possible.

\aph{} and the Bank should also be careful when selecting a $q$-tuple $\vew=\tuple{w_1,\dots,w_q}$ of complicated $p$-ary lattice terms. They should avoid that, for $i\in\set{1,\dots,q}$, the outmost operation symbol in $w_i$ is $\wedge$ and $w_i(\ves)$ is meet irreducible (or it has only few upper covers), and dually, and similarly for most of the subterms of $w_i$. 
In particular, $w_i(\ves)\in\set{0,1}$ should not happen.

To exemplify our ideas that come below, consider the (short) lattice term 
\begin{equation}  
x_4\Bigl(x_5+\Bigl(\bigl((x_1x_8+x_2x_3)\cdot(x_4x_5+x_3x_6)\bigr)+\bigl(x_2x_8+(x_3x_4)x_7\bigr)\Bigr)\Bigr);
\label{eqmnTltjrm}
\end{equation}
there are 15 \emph{occurrences} of variables in this term.
That is, if we represented this term by a binary tree in the usual way, then this three would have 15 leaves. 

Now, to choose a random term $w_1$, we can begin with a randomly chosen variable. Then, we iterate the following, say, a thousand times: after picking an occurrence of a variable in the already constructed term randomly (we denote this occurrence by $x_i$), selecting two of the $p$ variables, and  picking one of the two operations symbols, we replace $x_i$ by the meet or the join of the two variables selected, depending on which operations symbol has been picked. 

When choosing the two variables and the operation symbol mentioned above, we can exclude  that the replacement immediately ``cancels by the absorption laws''. (Or, at least, we have to be sure that this does not happen too often.) 
For example, it seems to be reasonable to forbid that  $x_6$ in \eqref{eqmnTltjrm} is replaced  by $x_3+x_7$. Although we can choose the occurrence mentioned in the previous paragraph according to the even distribution, it can be advantageous to go after a distribution that takes the \emph{depths} of the occurrences into account somehow. 

If $q=p$, which is recommended, then it is desirable that $\vew(\ves)$ should be far from $\ves$ and, in addition, each of the 
$w_1(\ves)$, \dots, $w_q(\ves)$ should be far from each other, from $0=0_L$, $1$, and from  $s_1$, \dots, $s_p$. By ``far'', we mean that the usual graph theoretical \emph{distance} in the Hasse diagram of $L$ or that of $L^q$ is larger than a constant.  Hence, while developing $w_1$ randomly, one can monitor $w_1(\ves)$ and interfere into the random process from time to time if necessary.

If $L$ is from \eqref{eqZhgRtBfktszg} or \eqref{eqdmzmgrYjHBbsgnwsztRtnk}, then $L$ is a semimodular lattice, so any two maximal chains of $L$ consist of the same number of elements. 
In this case, the above-mentioned distance  of $x,y\in L$ can be computed quite easily; see for example Cz\'edli, Powers, and White~\cite[equation (1.8)]{czpw}. Namely, the distance of $x$ and $y$ is
\begin{equation}
\distance x y =\length([x,x+y])+\length([y,x+y]).
\label{eqdlanghshsRw}
\end{equation}
Since any two maximal chains of $\Equ n$ are of the same size, it follows easily that $\length([x,x+y])$ is the difference of the number of $x$-blocks and the number of $(x+y)$-blocks, and similarly for 
$\length([x,x+y])$.

Several questions about the strategy remains open but future experiments with computer programs can lead to satisfactorily answers. However, even after obtaining good answers, the reliability of the above-described protocol would still remain the question of belief in some extent. This is not unexpected, since many modern cryptographic and similar protocols rely on the belief that certain problems, like factoring an integer or computing discrete logarithms, are hard.

Besides authentication, our method is also good for \emph{cryptography}. Assume that \aph{} and the Bank have previously agreed in $\ves$; see \eqref{eqphndThbnhnSpLSr}.
Then one of them can send a random $\vew$ to the other. 
They can both compute $\vew(\ves)$, see \eqref{eqszhnKhKypNxGqr}, but the Adversary cannot since even if he intercepts $\vew$, he does not know $\ves$. Hence, \aph{} and the Bank can use  $\vew(\ves)$  as the secret key of a classical cryptosystem like Vernam's. Such a secret key cannot be used repeatedly many times but \aph{} and the Bank can select a new $\vew$ and can get a new key $\vew(\ves)$ as often as they wish.

Next, we conjecture that  \aph{} can lock a \emph{commitment} $\ves$ by making $\vew(\ves)$ public.
To be more precise, the protocol is that there is a Verifier who chooses $\vew$, and then \aph{} computes  $\ver=\vew(\ves)$ 
with the Verifier's $\vew$ and makes this $\ver$ public. 
From that moment, \aph{} cannot change his commitment $\ves$, nobody knows what this $\ves$ is, but armed with $\vew$ and $\ver$, everybody can check \aph{} when he reveals $\ves$. Possibly, some stipulations should be tailored to $\ves$ and $\vew$ in this situation.


\begin{equation}
\lower  0.8 cm
\vbox{\tabskip=0pt\offinterlineskip
\halign{\strut#&\vrule#\tabskip=1pt plus 2pt&
#\hfill& \vrule\vrule\vrule#&
\hfill#&\vrule#&
\hfill#&\vrule#&
\hfill#&\vrule#&
\hfill#&\vrule#&
\hfill#&\vrule#&
\hfill#&\vrule\tabskip=0.1pt#&
#\hfill\vrule\vrule\cr
\vonal\vonal\vonal\vonal
&&\hfill$n$&&$\,4$&&$\,5$&&$\,6$&&$\,7$&&$8$&&$9$&
\cr\vonal\vonal
&& \hfill $|\Part n|$ &&$15$&&$52$&&$203$&&$877$&&$4\,140$&&$21\,147$&
\cr\vonal
&&\hfill$|\forall$8-sets$|$&&$6435$&&$7.53 \cdot 10^8$&&$6.22\cdot 10^{13}$&&$8.41\cdot 10^{18}$&&$2.13\cdot 10^{24}$&&$9.91\cdot 10^{29}$&
\cr\vonal
&&\hfill$|$tested$|$&&$100\,000$&&$10\,000$&&$10\,000$&&$6000$&&$1000$&&$284$&
\cr\vonal
&&\hfill$|$found$|$&&$89\,780$&&$7\,690$&&$7913$&&$5044$&&$848$&&$248$&
\cr
\vonal
&&\hfill \% &&$89.78$&&$76.90$&&$79.13$&&$84.01$&&$84.80$&&$90.19$&
\cr
\vonal\vonal\vonal\vonal
}} 
\label{tableD8gens}
\end{equation}

Finally, we have developed and used a computer program to see if there are sufficiently many $8$-element generating subsets and $n$-element generating sets of $\Part n$. 
This program, written in Bloodshed Dev-Pascal
v1.9.2 (Freepascal) under Windows 10 and partially in 
Maple V. Release 5 (1997),  is available from the author's website; see the list of publications there.
The results obtained with the help of this program are reported in Tables~\ref{tableD8gens} and \ref{tablenFpgFns}.  The first, \dots,  sixth rows in Tables~\ref{tableD8gens} give 
the size $n$ of the base set, 
the size of $\Part n$, 
the number of  8-element subsets of $\Part n$, 
the number of randomly selected 8-element subsets, 
the number of those selected 8-element subsets that generate $\Part n$, 
and the percentage of these generating 8-element subsets with respect to the number of the selected 8-element subsets, respectively. These subsets were selected independently according to the uniform distribution; a subset could be selected more than once.  
Table~\ref{tablenFpgFns} is practically the same but the $n$-element (rather than 8-element) subsets generating $\Part n$ are counted in it.
\begin{equation}
\lower  0.8 cm
\vbox{\tabskip=0pt\offinterlineskip
\halign{\strut#&\vrule#\tabskip=1pt plus 2pt&
#\hfill& \vrule\vrule\vrule#&
\hfill#&\vrule#&
\hfill#&\vrule#&
\hfill#&\vrule#&
\hfill#&\vrule#&
\hfill#&\vrule#&
\hfill#&\vrule\tabskip=0.1pt#&
#\hfill\vrule\vrule\cr
\vonal\vonal\vonal\vonal
&&\hfill$n$&&$\,4$&&$\,5$&&$\,6$&&$\,7$&&$8$&&$9$&
\cr\vonal\vonal
&& \hfill $|\Part n|$ &&$15$&&$52$&&$203$&&$877$&&$4\,140$&&$21\,147$&
\cr\vonal
&&\hfill$|\forall n$-sets$|$&&$1365$&&$2\,598\,960$&&$9.2\cdot 10^{10}$&&$7.73\cdot 10^{16}$&&$2.13\cdot 10^{24}$&&$2.33\cdot 10^{33}$&
\cr\vonal
&&\hfill$|$tested$|$&&$100\,000$&&$10\,000$&&$10\,000$&&$10000$&&$1000$&&$\phantom{f}$&
\cr\vonal
&&\hfill$|$found$|$&&$89\,780$&&$1430$&&$3918$&&$6811$&&$848$&&$\phantom{f}$&
\cr
\vonal
&&\hfill \% &&$89.78$&&$14.30$&&$39.18$&&$68.11$&&$84.80$&&$\phantom{f}$&
\cr
\vonal\vonal\vonal\vonal
}} 
\label{tablenFpgFns}
\end{equation}

Computing the last column of Table~\ref{tableD8gens} took 73 hours for a desktop computer
with AMD Ryzen 7 2700X Eight-Core Processor 3.70 GHz; this explains that no more 8-element subsets have been tested for Table~\ref{tableD8gens} and the last column of Table~\ref{tablenFpgFns} is partly missing. After computing the columns for $n=4$ and $n=5$ in Tables~\ref{tableD8gens} and \ref{tablenFpgFns}, we expected that the number in the percentage row (the last row) would decrease as $n$ would decrease as $n$ grows. To our surprise, the opposite happened. 
Based on these two tables, we guess that $p=n$ should be and even $p=8$ could be appropriate in the protocol if $n=273$ and
$L$ is taken from \eqref{eqZhgRtBfktszg}.

\subsection*{Chronology and comparison, added on \chronologydatum}
The first version of the present paper
was uploaded 
to \texttt{https://arxiv.org/abs/2004.14509}
on April 29, 2020.
A related \emph{second paper} dealing with direct products rather than direct powers
was completed
and  uploaded to
\texttt{http://arxiv.org/abs/2006.14139}
on June 25, 2020.  
(This second paper pays no attention to  authentication and cryptography.)  
The present paper corrects few typos and minor imperfections but it is not significantly different from its April 29, 2020 version. Although a particular case of the second paper also tells something on four-generation of direct powers of finite partition lattices, the present paper, yielding larger exponents and paying attention to $(1+1+2)$-generation, tells more. For example, while the four-generability of $\Part {2020}^{10^{127}}$ is almost explicit in the second paper and the maximum we can extract from \emph{that} paper is approximately the  four-generability
 of  $\Part {2020}^{10^{604}}$, the last column of Table \eqref{tablerdDbsg} in the \emph{present} paper guarantees a significantly larger exponent, $5.5194232\cdot10^{3893}$. (The corresponding value,  $5.52\cdot10^{3893}$, from Table \eqref{tablerdDbsg} was obtained by rounding up.)

\end{document}